\documentclass[12pt]{amsart}

\usepackage[utf8]{inputenc}
\usepackage[francais]{babel}
\usepackage{hyperref}
\usepackage{comment}

\usepackage{amssymb,amsthm,amsmath,mathrsfs,amssymb}
\usepackage{euler}

\relax

\renewcommand{\leq}{\leqslant}
\renewcommand{\geq}{\geqslant}

\makeatletter
\def\paragraph{\@startsection{paragraph}{4}%
  \z@\z@{-\fontdimen2\font}%
  {\normalfont\bfseries}}
\makeatother

\newcommand{\GL}{GL}
\newcommand{\Z}{\mathbb{Z}}

\newcommand{\R}{\mathbb{R}}
\newcommand{\C}{\mathbb{C}}
\newcommand{\modzero}{\,\left(\mathrm{mod}\,0\right)}

\newcommand*\colvec[3][]{
    \begin{pmatrix}\ifx\relax#1\relax\else#1\\\fi#2\\#3\end{pmatrix}
}

\newcommand{\norm}[1]{\left\lVert#1\right\rVert}
\newtheorem{thm}{Théorème}
\newtheorem{lem}{Lemme}

\theoremstyle{definition}

\newtheorem{rem}{Remarque}

\newtheorem{fait}{Fait}
\title{Inexistence de pavages mesurables invariants par un réseau dans un espace homogène d'un groupe de Lie simple}
\author{Félix Lequen}
\address{Laboratoire AGM -- CY Cergy Paris Université}
\email{felix.lequen@cyu.fr}
\usepackage{fullpage}
\begin{document}
	\maketitle

\section{Introduction}

\begin{abstract}
	On démontre qu'un espace homogène d'un groupe de Lie presque simple n'admet pas de pavage mesurable invariant par un réseau du groupe de Lie. Ceci constitue un raffinement du théorème d'ergodicité de Howe-Moore. 
\end{abstract}

Soit $\Gamma$ un groupe dénombrable agissant sur un espace mesuré $\sigma$-fini $(X, \mathcal{B}, \mu)$, tel que $\Gamma$ \emph{quasi-préserve} la mesure $\mu$ : pour tout $\gamma \in \Gamma$ et tout $A \in \mathcal{B}$, $\mu(\gamma A) = 0$ si et seulement si $\mu(A) = 0$. On dit que l'action de $\Gamma$ est \emph{ergodique} si tout ensemble $A \in \mathcal{B}$ qui est invariant par $\Gamma$ vérifie soit $\mu(A) = 0$, soit $\mu(X\setminus A) = 0$. Toute action ergodique d'un groupe sur un espace sans atome est \emph{conservative} : pour tout ensemble $A \in \mathcal{B}$ tel que $\mu(A) > 0$, il existe $\gamma \in \Gamma \setminus \{e\}$ tel que $\mu(A \cap \gamma A) > 0$. Ici et dans la suite, on a noté $e$ l'élément neutre d'un groupe

Le cas qui nous intéressera ici est celui des actions de réseaux sur les espaces homogènes de groupe de Lie presque simples, c'est-à-dire d'algèbre de Lie simple. Le théorème de Howe-Moore ci-dessous affirme alors l'ergodicité de l'action, dès lors que le stabilisateur est non-compact. 
\begin{thm}[Howe-Moore \cite{zimmer2013ergodic}]
	Soit $G$ un groupe de Lie connexe presque simple à centre fini. Soit $\mu$ une mesure de Haar de $G$ et $\Gamma$ un réseau de $G$.
	
	Soit $H$ un sous-groupe d'adhérence non compacte de $G$.
	
	Alors si $A \subset G$ est $H$-invariant à droite et $\Gamma$-invariant à gauche, on a $\mu(A) = 0$ ou $\mu(G\setminus A) = 0$.
\end{thm}

Dans cette note, on s'intéresse à une condition plus forte que l'ergodicité, à savoir l'inexistence de pavages mesurables non triviaux invariants par une action de groupe, c'est-à-dire d'éléments $P \in \mathscr{B}$ (appelé \emph{pavés}) tels que pour tout $\gamma \in \Gamma$, $\gamma P \cap P$ est de mesure nulle, ou $\gamma P \Delta P$ est de mesure nulle. Il est possible qu'une action conservative (et même ergodique) d'un groupe dénombrable admette un pavage mesurable non trivial, c'est-à-dire de pavé de mesure non nulle ou pleine. Par exemple, considérons $\Gamma$ un groupe dénombrable agissant de façon ergodique sur $(X, \mathscr{B}, \mu)$ et $G$ un groupe discret dénombrable quelconque, muni de la mesure de comptage, et l'action de $\Gamma \times G$ sur l'espace $X \times G$ définie par
\[(\gamma, g) \cdot (x, h) := (\gamma x, gh)\]
Cette action est ergodique, et pourtant l'ensemble $X \times \{e\}$ définit un pavage mesurable non-trivial de $X \times G$.

Une condition qui entra\^ine l'inexistence de pavages mesurables non triviaux est celle de \emph{mélange faible}, ou \emph{double ergodicité}, c'est-à-dire lorsque l'action de $\Gamma$ sur $(X, \mathscr{B}, \mu)$ est telle que l'action diagonale de $\Gamma$ sur $(X \times X, \mathcal{B} \otimes \mathcal{B}, \mu \otimes \mu)$ est ergodique. Dans ce cas, si $P$ est un pavé définissant un pavage mesurabl de $X$, alors l'ensemble $A \in \mathcal{B} \otimes \mathcal{B}$ des éléments appartenant à des pavés différents est invariant par l'action diagonale de $\Gamma$, ce qui entra\^ine que $P$ est de mesure nulle ou pleine. 

Cette condition est par exemple réalisée pour l'action par homographie de $SL_2(\Z)$ (ou d'un autre réseau) sur la droite projective réelle $\mathbb{P}^1(\R)$. L'action s'identifie à l'action par translation du réseau $SL_2(\Z)$ sur $SL_2(\R)/B$ où $B$ est le groupe des matrices triangulaires supérieures de $SL_2(\R)$, et l'action diagonale à l'action par translation sur $SL_2(\R)/A$, où $A$ est le groupe des matrices diagonales de $SL_2(\R)$. Ces deux actions sont ergodiques par le théorème de Howe-Moore — cet exemple se généralise en fait à l'action d'un réseau sur une variété de drapeaux d'un groupe de Lie simple. L'action de $SL_2(\Z)$ sur $\left(\mathbb{P}^1(\R) \times \mathbb{P}^1(\R) \setminus (\text{diagonale})\right) \simeq SL_2(\R)/A$ n'est en revanche pas elle-m\^eme faiblement mélangeante, en raison de l'existence du birapport qui définit une fonction invariante. Comme nous allons le voir, un argument semblable à celui de la preuve du théorème de Howe-Moore permet de montrer qu'il n'y a pas pour autant de pavage mesurable non trivial invariant par un réseau :

\begin{thm}
Soit $G$ un groupe de Lie connexe presque simple à centre fini. Soit $\mu$ une mesure de Haar de $G$ et $\Gamma$ un réseau de $G$.

Soit $H$ un sous-groupe d'adhérence non compacte de $G$.

Soit $P$ une partie mesurable de $G$, $H$-invariante à droite et telle que pour tous $\gamma_1, \gamma_2 \in \Gamma$, ou bien $\gamma_1P = \gamma_2P \modzero$, ou bien $\gamma_1P \cap \gamma_2P = \varnothing \modzero$.
Alors ou bien $P$ est de mesure nulle, ou bien $P$ est de mesure pleine. 
\end{thm}

Ici, on a noté $A = B \modzero$ lorsque $A = B$ à un ensemble de mesure nulle près ; on utilisera une notation analogue pour d'autres propriétés ensemblistes.

\paragraph{Remerciements. } Je remercie mon encadrant, Bertrand Deroin, pour m'avoir guidé et encouragé tout au long de la préparation de cette note, et Rémi Boutonnet pour sa relecture et son intérêt pour cette question.
\section{Preuve du théorème}
Soit $B$ le complémentaire de la réunion des $\gamma P$, pour $\gamma \in\Gamma$. C'est un ensemble $\Gamma$-invariant à gauche et $H$-invariant à droite, à des ensembles de mesure nulle près, donc il est de mesure nulle ou pleine. Si $B$ est de mesure pleine, alors $P$ est de mesure nulle. Dans la suite, on suppose donc que $B$ est de mesure nulle.

On note $\mu$ une mesure de Haar à gauche sur $G$. La mesure de Haar invariante à gauche est aussi invariante à droite, de sorte qu'on parle simplement de mesure de Haar sans plus de précision. On note $\nu$ la mesure $G$-invariante à droite sur $\Gamma\backslash G$ associée à $\mu$ et $\mathfrak{g}$ l'algèbre de Lie de $G$.
Pour tout $g \in G$, soit $\tilde{A}(g) := \bigcup_{\gamma \in \Gamma} \gamma(P\cap Pg)$ et $A(g)$ l'image de $\tilde{A}(g)$ dans $\Gamma\backslash G$. Enfin, soit 
\begin{align*}S &:= \{g \in G\,:\,\nu(A(g))= \nu(\Gamma\backslash G)\}\\
&= \{g \in G\,:\,\mu(G\setminus\tilde{A}(g)) = 0\}\\
&= \{g\in G\,:\,P = Pg \modzero\}\end{align*}

Justifions la dernière égalité. Une des deux inclusions est claire ; montrons l'autre. Soit $g \in G$ tel que $\tilde{A}(g)$ soit de mesure pleine. Alors l'ensemble $P\setminus \left(P \cap Pg\right)$ est de mesure nulle, puisqu'il rencontre $\tilde{A}(g)$ en un ensemble de mesure nulle : en effet, si $\gamma \in \Gamma$ est tel que $\gamma(P\cap Pg) \subset \gamma P$ rencontre $P\setminus \left(P \cap Pg\right) \subset P$ en un ensemble de mesure strictement positive, alors on a nécessairement $\gamma P = P \modzero$. Mais alors $\gamma(P \cap Pg) = P \cap Pg \modzero$. Donc $P \subset Pg \modzero$. En considérant $\tilde{A}(g^{-1}) = \tilde{A}(g)g^{-1}$, on a aussi $P \subset Pg^{-1} \modzero$, d'où l'on déduit le résultat.

On a les propriétés élémentaires suivantes de ces ensembles $A(g)$ ($g \in G$) :

\begin{fait}
\label{fait-inv}
L'ensemble $S$ est un sous-groupe de $G$. De plus, si $g \in G$ et $h \in S$, on a $A(hg) = A(g) \modzero$ et $A(g)h = A(gh) \modzero$. Enfin, $A(e) = \Gamma\backslash G$.
\end{fait}  

\begin{lem}
L'application $g \mapsto \nu(A(g))$ est continue. En particulier, le sous-groupe $S$ de $G$ est fermé.
\end{lem}
\begin{proof}
Soit $\mathscr{F} \subset G$ un domaine fondamental borélien de $\Gamma$. Soit $\mathscr{X}$ un ensemble de translatés de $P$ à gauche par $\Gamma$ tel que les éléments de $\mathscr{X}$ sont deux à deux disjoints, à ensembles de mesure nulle près, et leur réunion est de mesure pleine. On a donc $\tilde{A}(g) = \bigcup_{Q \in \mathscr{X}} (Q \cap Qg) \modzero$
Alors \[g \mapsto \nu(A(g)) = \mu\left(\mathscr{F} \cap \tilde{A}(g)\right) = \sum_{Q \in \mathscr{X}}\mu\left(\mathscr{F} \cap Q \cap Qg\right)\]

Pour tout $Q \in \mathscr{X}$, l'application $g \mapsto \mu\left(\mathscr{F} \cap Q \cap Qg\right) = \mu\left(Q \cap \left(\mathscr{F} \cap Q\right)g^{-1}\right)$ est continue. En effet, pour tous $g, h \in G$, on a :
\[\left|\mu\left(Q \cap \left(\mathscr{F} \cap Q\right)g^{-1}\right) - \mu\left(Q \cap \left(\mathscr{F} \cap Q\right)h^{-1}\right)\right| \leq \norm{1_{(\mathscr{F} \cap Q)g^{-1}} - 1_{(\mathscr{F} \cap Q)h^{-1}}}_{L^1(G, \mu)}\]qui tend vers $0$ quand $g$ tend vers $h$ par un raisonnement classique par densité et invariance de $\mu$ (ici, $1_X$ désigne bien entendu la fonction indicatrice du borélien $X$).

Finalement, par convergence dominée, l'application $g \mapsto \sum_{Q \in \mathscr{X}} \mu\left(\mathscr{F} \cap Q \cap Qg^{-1}\right)$ est bien continue.
\end{proof}

Dans la suite, on note $\mathfrak{s} \subset \mathfrak{g}$ l'algèbre de Lie de $S$.

\begin{lem}[Phénomène de Mautner]
\label{mautner}
Soient $a \in S$ et $g\in G$ tel que $a^nga^{-n}$ converge vers l'élément neutre $e$ de $G$ quand $n \longrightarrow \infty$. Alors $g$ appartient à $S$.
\end{lem}
\begin{proof}
En effet, on a, en utilisant le fait \ref{fait-inv} :
\[\nu(A(g)) = \nu(A(g)a^{-n}) = \nu(A(a^nga^{-n})) \longrightarrow_{n \longrightarrow \infty} \nu(\Gamma\backslash G)\]
\end{proof}

\begin{lem} 
Supposons qu'il existe $a \in S$ un élément hyperbolique non trivial. Alors $S = G$.
\label{hyperbolique-inv}
\end{lem}
Ici on dit qu'un élément $a$ est \emph{hyperbolique} s'il existe $X \in \mathfrak{g}$ avec $\mathrm{ad}\,X$ diagonalisable sur $\R$ tel que $a = \exp(X)$.
\begin{proof}
Soit $a = \exp(X)$ un élément hyperbolique non trivial. Pour $\lambda\in \R$, soit $\mathfrak{g}^\lambda \subset \mathfrak{g}$ le sous-espace propre de $\mathrm{ad}\,X$ correspondant à la valeur propre $\lambda$. Soient $\mathfrak{g}^{+} = \bigoplus_{\lambda > 0} \mathfrak{g}^\lambda$ et $\mathfrak{g}^{-} = \bigoplus_{\lambda < 0} \mathfrak{g}^\lambda$. 

\begin{fait}
La sous-algèbre de Lie engendrée par $\mathfrak{g}^{+} \cup \mathfrak{g}^{-}$ est égale à $\mathfrak{g}$.
\end{fait}
\begin{proof}En effet, si l'on note $\mathfrak{l}$ la sous-algèbre de Lie qu'ils engendrent, la relation $[\mathfrak{g}^\lambda, \mathfrak{g}^\mu] \subset \mathfrak{g}^{\lambda + \mu}$, valable pour tous $\lambda, \mu \in \R$, montre que $\mathfrak{l}$ est un idéal. On a donc $\mathfrak{l} = \mathfrak{g}$ ou $\mathfrak{l} = 0$. Dans ce dernier cas, on a $\mathrm{ad}\,X = 0$, donc $X = 0$ par simplicité, et ainsi $a = 1$, ce qui est faux. Donc $\mathfrak{l} = \mathfrak{g}$.
\end{proof}

\begin{fait}
On a $\mathfrak{g}^+ \cup \mathfrak{g}^{-} \subset \mathfrak{s}$.
\end{fait}
\begin{proof}
Soit $Y \in \mathfrak{g}^{\lambda}$ et notons $g = \exp(Y)$. Alors \[a^nga^{-n} = \exp((\mathrm{Ad}\,a^n)Y) = \exp(e^{n(\mathrm{ad}\,X)}Y) = \exp(e^{n \lambda}Y)\]
Ainsi, lorsque $n \longrightarrow \pm\infty$ (suivant le signe de $\lambda$), $a^nga^{-n} \longrightarrow e$. Par le phénomène de Mautner (lemme \ref{mautner}), on en déduit le résultat, puisque $a \in S$.
\end{proof}
Par les deux faits précédents, on a donc $\mathfrak{s} = \mathfrak{g}$, d'où $S = G$ puisque $G$ est connexe.
\end{proof} 

\begin{lem}
Supposons qu'il existe $u \in S$ un élément unipotent non trivial. Alors $S = G$.
\end{lem}
Ici on dit qu'un élément $u$ est \emph{unipotent} s'il existe $X \in \mathfrak{g}$ avec $\mathrm{ad}\,X$ nilpotent tel que $u = \exp(X)$.
\begin{proof}
Soit $u = \exp(X)$ avec $X \in \mathfrak{g}$ tel que $\mathrm{ad}\,X$ soit nilpotent. Soit $L = \{\exp(tX),\,t\in\R\} \subset S$. Alors $L$ est d'adhérence non compacte (car $\mathrm{Ad}\,S'$ est d'adhérence non compacte).
Par le théorème de Jacobson-Morozov (voir Helgason \cite[théorème 7.4]{helgason1979differential}), il existe $H, Y \in \mathfrak{g}$
tels que
\[[H, X] = 2X\text{,}\;[H, Y] = -2Y\text{ et }[X, Y] = H\]
Ainsi $\R H \oplus \R X \oplus \R Y$ est une sous-algèbre de Lie isomorphe à $\mathfrak{sl}_2(\R)$. La structure des représentations de $\mathfrak{sl}_2(\R)$ montre alors que $\mathrm{ad}\,H$ est diagonalisable sur $\R$. Notons que l'on a, pour tous $\lambda, t\in \R$ :
\[\exp(\lambda H)\exp(t X)\exp(- \lambda H) = \exp(te^{\lambda(\mathrm{ad} H)}X) = \exp(te^{2\lambda}X)\]
 Par le fait \label{fait-normalisateur} ci-dessous, on a donc $e^{\lambda H}\in S$ pour tout $\lambda \in \R$. Par le lemme \ref{hyperbolique-inv}, on a donc $S = G$.
\begin{fait}\label{fait-normalisateur}
	Soit $T \subset S$ un sous-groupe d'adhérence non compacte de $G$, et $L$ un sous-groupe de $G$. On suppose que pour tout $g \in L$, on a $gTg^{-1} \subset S$. Alors la composante neutre $L^\circ$ de $L$ est incluse dans $S$.
\end{fait}
\begin{proof}
On a pour tout $h \in T$ et $g \in L$, grâce au fait \ref{fait-inv} :
	\[A(g)h = A(gh) = A(ghg^{-1}g) = A(g) \modzero\]
	Donc par le théorème de Howe-Moore, on a $\nu(A(g)) \in \{0, 1\}$. Or par continuité, on a $\nu(A(g)) > 0$ dans un voisinage de l'identité de $L$, d'où le résultat. 
\end{proof}

\end{proof}

\begin{lem}
	\label{groupe-tout}
Le groupe $S$ contient soit un élément unipotent, soit un élément hyperbolique, et donc $S = G$.
\end{lem}
\begin{proof}
Considérons le groupe $\mathrm{Ad}_G(S)$, qui est non-compact, puisque $S$ l'est. Par le lemme \ref{lemme-elliptique} ci-dessous, il existe $g \in S$ tel que $\mathrm{Ad}\,g$ ne soit pas elliptique. Écrivons, par la décomposition de Jordan-Chevalley réelle de $\mathrm{Ad}_G(G)$, $\mathrm{Ad}\,g = ehu$ avec $e$ elliptique, $h$ hyperbolique et $u$ unipotents qui commutent avec $g$. On a donc $h \neq e$ ou $u \neq e$. Alors $h$ et $u$ s'écrivent respectivement $e^{\mathrm{ad}\,H}$ et $e^{\mathrm{ad}\,N}$ où $H, N \in \mathfrak{g}$, $\mathrm{ad}\,H$ est diagonalisable sur $\R$ et $\mathrm{ad}\,N$ est nilpotente \cite[chapitre IX, lemme 7.3 et sa preuve]{helgason1979differential}, et on a $H \neq 0$ ou $N \neq 0$.

Soit $t \in \R$ et $A \in \{H, N\}$. Comme $e^{\mathrm{ad}\,tA}$ commute avec $\mathrm{Ad}\,g$, le commutateur $e^{-tA}ge^{tA}g^{-1}$ appartient au centre de $G$, qui est fini. Donc pour $t$ assez petit, $e^{tA}$ et $g$ commutent, donc le groupe $\{e^{tA}\,:\,t \in \R\}$ normalise le groupe engendré par $g$. Donc par le fait \ref{fait-normalisateur}, $A \in \mathfrak{s}$. D'où le résultat par les deux lemmes précédents.
\end{proof}

\begin{lem}
	\label{lemme-elliptique}
	Soit $G \subset \GL_n(\C)$ un sous-groupe fermé dont tous les éléments sont diagonalisables sur $\C$ à valeurs propres de module $1$.
	Alors $G$ est compact.
\end{lem}
\begin{proof}
	Soit $\mathfrak{g} \subset M_n(\C)$ l'algèbre de Lie de $G$ et $\mathfrak{r}$ le radical de $\mathfrak{g}$. Remarquons que pour tout $X \in \mathfrak{g}$, $e^X$ est diagonalisable, ce qui entraîne que $X$ est diagonalisable, par exemple en utilisant la décomposition de Jordan.
	
	Comme $\mathfrak{r}$ est résoluble, on peut triangulariser simultanément ses éléments sur $\C$. Cela implique que si $X, Y \in \mathfrak{r}$, $[X, Y]$ est triangulaire supérieure stricte dans une base adéquate de $\mathfrak{g}$, donc nilpotent. Comme $[X, Y]$ est diagonalisable sur $\C$, on en déduit que $[X, Y] = 0$. Ainsi, $\mathfrak{r}$ est abélienne.
	
	Si $X \in \mathfrak{r}$ et $Y \in \mathfrak{g}$, on a $[X, Y] \in \mathfrak{r}$ donc $[X, [X, Y]] = 0$, donc $\mathrm{ad}\,X$ induit un endomorphisme nilpotent de $\mathfrak{g}$, donc comme $X$ est diagonalisable sur $\C$, $\mathrm{ad}\,X$ est diagonalisable, donc nul. Donc $\mathfrak{r}$ est inclus dans $\mathfrak{z}(\mathfrak{g})$, le centre de $\mathfrak{g}$. 
	
	Ceci entraîne que l'algèbre de Lie $\mathfrak{g}$ est réductive, et donc $\mathfrak{g} = [\mathfrak{g}, \mathfrak{g}] \oplus \mathfrak{z}(\mathfrak{g})$, où l'algèbre $[\mathfrak{g}, \mathfrak{g}]$ est semi-simple. On peut le montrer directement : la représentation adjointe induit une représentation de $\mathfrak{g}/\mathfrak{r}$ sur $\mathfrak{g}$ qui est complètement réductible car $\mathfrak{g}/\mathfrak{r}$ est semi-simple. Donc le sous-espace invariant $\mathfrak{z}(\mathfrak{g})$ admet un supplémentaire invariant, isomorphe à $\mathfrak{g}/\mathfrak{r}$, donc semi-simple.
	
	Finalement, on déduit de ceci que $G^\circ \subset \mathrm{GL}_n(\C)$ agit de façon complètement réductible (voir Serre \cite[Theorem 5.1]{serre1992lie}). On utilise ici l'hypothèse que $X$ est diagonalisable si $X \in \mathfrak{z}(\mathfrak{g})$. Or on a le lemme suivant.
	
	\begin{lem}[\cite{benoist}]
		Soit $n \geq 1$ et $H \subset \mathrm{GL}(n, \C)$ un sous-groupe de matrices agissant de façon irréductible sur $\C^n$ et tel que tout les éléments de $H$ aient toutes leurs valeurs propres complexes de module $1$. Alors $H$ est relativement compact.
	\end{lem}
	\begin{proof}
		Soit $A$ l'algèbre engendrée par $H$. Soit $I$ le noyau de la forme bilinéaire symétrique $(a, b) \mapsto \mathrm{tr}(ab)$. C'est clairement un idéal de $A$. Si $a \in I$, on a $0 = \mathrm{tr}(a) = \mathrm{tr}(a^2) = \ldots = \mathrm{tr}(a^n)$, donc $a$ est nilpotent. Si $I \neq 0$, par le théorème d'Engel, on déduit que $J$, l'ensemble des vecteurs annulés par tous les éléments de $I$, est non nul. Or $J$ est un sous-espace invariant par $A$. Donc $I = 0$, et la forme $(a, b) \mapsto \mathrm{tr}(ab)$ est non dégénérée.
		
		Soit maintenant $(g_i)$ une base de $A$ formée d'éléments de $H$ et $(e_i)$ la base duale pour la forme $(a, b) \mapsto \mathrm{tr}(ab)$. Si $g \in H$, on a ainsi $g = \sum_i \mathrm{tr}(gg_i)e_i$. Mais $gg_i \in H$, donc $|\mathrm{tr}(gg_i)| \leq n$. Donc $H$ est borné, et donc relativement compact en prenant l'inverse.
	\end{proof}
	 Comme $G$ permute les sous-espaces irréductibles de $G^\circ$, il existe un sous-groupe d'indice fini $G'$ de $G$ qui stabilise chacun de ces sous-espaces. Le lemme précédent entraîne alors que $G'$ est relativement compact. Donc $G$ est aussi relativement compact, et donc compact.
\end{proof}
\begin{proof}[Conclusion de la preuve]
Par le lemme \ref{groupe-tout}, on a $S = G$, ce qui entraîne que pour tout $g \in G$, $P = Pg\modzero$. Donc $P$ est de mesure pleine.
\end{proof}

\begin{rem}
	Le lemme \ref{lemme-elliptique} n'est pas vrai si on ne suppose pas le groupe fermé (Bass \cite{bass}).
\end{rem}

\bibliographystyle{plain}
\bibliography{K3Levi}
\end{document}